\documentclass[a4paper,12pt]{amsart}

\usepackage{tikz}

\newcommand{\PP}{\mathbb{P}}

\newif\ifdetails
\detailstrue
\newcommand{\DETAIL}[1]%
{\ifdetails\par\fbox{\begin{minipage}{0.9\linewidth}\textit{Detail:}
      #1\end{minipage}}\par\fi}
\newcommand{\TODO}[1]%
{\ifdetails\par\fbox{\begin{minipage}{0.9\linewidth}\textbf{TODO:}
      #1\end{minipage}}\par\fi}

\newtheorem{lemma}{Lemma}

\newtheorem{theorem}[lemma]{Theorem}

\theoremstyle{remark}

\newcommand{\old}[1]{{}}

\title{A degree condition for diameter two orientability of graphs }

\author{\'Eva Czabarka, Peter Dankelmann, L\'aszl\'o A. Sz\'ekely }
\address{\'Eva Czabarka and L\'aszl\'o A. Sz\'ekely\\ Department of Mathematics \\ University of South Carolina \\ Columbia, SC 29208 \\ USA, and Visiting Professor\\ Department of Pure and Applied Mathematics \\ University of Johannesburg \\ P.O. Box 524\\ Auckland Park, Johannesburg 2006 \\ South Africa}
\email{\{czabarka,szekely\}@math.sc.edu }
\thanks{The second author was supported in part by 
the National Research Foundation of South Africa, grant number 103553,  
the third author was supported by the NSF DMS, grant number  1600811.}
\address{Peter Dankelmann\\ Department of Pure and Applied Mathematics \\ University of Johannesburg \\ P.O. Box 524\\ Auckland Park, Johannesburg 2006 \\ South Africa}
\email{pdankelmann@uj.ac.za}
\subjclass[2010]{Primary 05C12; secondary 05C20, 05C35}
\keywords{graph, distance, diameter, orientation, oriented diameter}

\begin{document}

\begin{abstract}
For $n \in \mathbb{N}$ let $\delta_n$ be the smallest value
such that every graph of order $n$ and minimum degree at least
$\delta_n$ admits an orientation of diameter two. We show that
$\delta_n=\frac{n}{2} + \Theta(\ln n)$. 
\end{abstract}

\maketitle

\section{Introduction}

Let $G$ be a connected graph. An orientation of $G$ is a digraph 
obtained from $G$ by assigning a direction to every edge of $G$. 
The diameter of strong digraph, i.e., of a digraph in which there is a 
directed path between any two vertices, is the maximum distance between any 
two vertices. The oriented diameter $\overrightarrow{{\rm diam}}(G)$ 
of $G$ is the minimum diameter among all strong orientations of $G$. 
A bridge is an edge of a connected graph whose removal
renders the graph disconnected. 

The well-known Robbin's Theorem \cite{Rob1939} states that every 
bridgeless connected graph has a strong orientation.  For many applications it is 
desirable to have a strong orientation whose diameter is as small as possible. 
The natural question if bridgeless graphs of small diameter necessarily 
have orientations of small diameter was answered in the
affirmative by Chv\'{a}tal and Thomassen \cite{ChvTho1978}, who 
showed the existence of a (least) function $f: \mathbb{N} \rightarrow \mathbb{N}$ 
such that every bridgeless graph of diameter $d$
has an orientation of diameter at most $f(d)$. 
Chv\'{a}tal and Thomassen \cite{ChvTho1978} showed that 
$f(2)=6$, and that  
$\frac{1}{2}d^2 + d \leq f(d) \leq 2d^2 + 2d$ for all $d\in \mathbb{N}$. 
The value of  $f(3)$  is not known, but 
Egawa and Iida \cite{EgaIid2007}, and independently Kwok, Liu and 
West \cite{KwoLiuWes2010} gave the estimate $9 \leq f(3) \leq 11$. 

Upper bounds on $\overrightarrow{{\rm diam}}(G)$ are known 
in terms of other graph parameters, such as 
minimum degree \cite{BauDan2015, Sur2017}, 
maximum degree \cite{DanGuoSur2018}, 
and domination parameters \cite{LatKur2012, FomMatPriRap2004}. 
The oriented diameter has also been investigated for graphs from 
various special graph classes  
\cite{EggNob2009, FomMatPriRap2004, FomMatRap2004, KohTan1996, KohTan1996-2, KohTay1999, 
KohTay2001, GutYeo2002, HuaYe2007}. 
For a survey of results on the diameter of orientations of
graphs up to about year 2000 see \cite{KohTay2002}. 

Of special interest is the question which graphs have an orientation
of diameter two, the smallest possible diameter of an orientation of 
any graph on more than one vertex. It was shown by Chv\'{a}tal and 
Thomassen \cite{ChvTho1978} that the decision problem if a given input 
graph has an orientation of diameter two is NP-complete. 
Hence sufficient conditions for the existence of an orientation
of diameter two are of interest. 
The smallest value $m_n$, so that every graph of order $n$ and size 
at least $m_n$ admits an orientation of diameter two was conjectured
in \cite{KohTay2002} to be $\binom{n}{2}-n+5$ and this conjecture
was proved in \cite{CocCzaDanSze-manu}.
In this paper we present a sufficient condition for an 
orientation of diameter two in terms of  
minimum degree. For $n\in \mathbb{N}$, $n\geq 3$, define 
$\delta_n$ to be the smallest value such that every graph of order
$n$ and minimum degree at least $g(n)$ admits an orientation
of diameter two. 
In view of the fact that for a graph $G$ of order $n$ we need the minimum degree 
to be at least $\frac{n-1}{2}$ for $G$ to guarantee that the diameter 
is not more than two, it is clear that 
$\delta_n \geq \frac{n-1}{2}$.  
It is the aim of this paper to show that  
\[ \delta_n = \frac{n}{2} + \Theta(\ln n). \]

The notation we use in this paper is as follows. $G$ always denotes a 
connected graph on $n$ vertices. The neighborhood of a vertex $v$ of $G$,
i.e., the set of all vertices adjacent to $v$ is denoted by $N_G(x)$, and the degree of $v$ is $|N_G(v)|$.
In a digraph $D$, $N^+_D(v)$ is the set of vertices reached from $v$ by an arc, and $N^-_D(v)$ is the set of vertices that reach $v$ by an arc.
The minimum degree $\delta(G)$ of $G$ is the smallest of 
the degrees of the vertices of $G$. 
The distance from a vertex $u$ to a vertex $v$, i.e., the minimum
length of a $(u,v)$-path in a graph or digraph is 
denoted by $d(u,v)$.  By 
a $2$-path we mean a path of length two. In a given probability space 
we write $\PP[X]$ for the probability of an event $X$, and $\mu$ for the
expected value of a random variable.

\section{Results}

\begin{theorem}  \label{theo:sufficent-for-diam2-orient}
Let $G$ be a graph on $n$ vertices. If 
\[ \delta(G) \geq \frac{n}{2} +\frac{\ln n}{\ln (4/3)} 
   = \frac{1}{2}n + (3.476...) \ln n, \]
then $\overrightarrow{{\rm diam}}(G)=2$.
\end{theorem}

\begin{proof} 
Assume that the minimum degree of the graph  $G$ is at least $n/2+f(n)$, where 
$f(n)=\frac{1}{\ln (4/3)} \ln n$. 
Fix two arbitrary vertices of $G$, $x$ and $y$. By the inclusion-exclusion formula, 
$|N_G(x)\cap N_G(y)|=|N_G(x)|+|N_G(y)|-|N_G(x)\cup N_G(y)|\geq n+2f(n)-n=2f(n)$. 
 
Assign one of the two possible orientations to every edge of $G$ randomly and 
independently with probability $1/2$. 
Define the random variable $X_{xy}$ to be $1$ if we have 
no directed 2-path from $x$ to $y$, and $0$ if there is such 
a path. Clearly, the expected value of $X_{xy}$ is  
$\PP[X_{xy}=1] = \bigl(\frac{3}{4}\bigl)^{|N(x) \cap N(y)|} 
             \leq \bigl(\frac{3}{4}\bigl)^{2f(n)}$. 
Therefore, for the expected number $\mu$ of ordered pairs
$x,y$ with no directed 2-paths from
$x$ to $y$, we have 
\begin{equation}  \label{eq:expected-value}
\mu=\sum_{x,y \in V(G), x\neq y}  \PP[X_{xy}]  \leq 
     n(n-1)\Bigl(\frac{3}{4}\Bigl)^{2f(n)}
     <  n^2 \Bigl(\frac{3}{4}\Bigl)^{2f(n)}.  
\end{equation}     
Let $X = \sum_{x,y \in V(G), x\neq y} X_{xy}$ be the number of
distinct ordered pairs of vertices that are not joined by a 
directed 2-path.  
Then $X=0$ means that a random orientation has diameter
two.  If $f(n)$ is selected so large that  
\[ n^2  \Bigl(\frac{3}{4}\Bigl)^{2f(n)} \leq 1, \]
then if follows from \eqref{eq:expected-value} that 
$\mu <1$, and so at least one
orientation satisfies $X=0$, implying that 
$\overrightarrow{{\rm diam}}(G)=2$. 
Easy calculations shows that this is indeed the case if 
$f(n)\geq \frac{1}{\ln(4/3)}\ln n$.
\end{proof}

\begin{theorem} \label{theo:sharpness-example}
There exists an infinite family of graphs $G$ of order $n$ and
\[ \delta(G) \geq \frac{n}{2} + \frac{\ln n}{2\ln(27/4)}=\frac{n}{2} + (0.2618...){\ln n}.\]
which do not have an orientation of diameter $2$. 
\end{theorem}

\begin{proof}
For $k \in \mathbb{N}$ we construct the graph $G_k$ as follows. 
Let $N = {3k \choose k}$. Let $H_1$, $H_2$ and $H_3$ be disjoint
copies of the complete graph $K_N$, $K_{3k}$ and $K_N$, 
respectively. Let $G_k$ be obtained from the disjoint union of  
$H_1$, $H_2$ and $H_3$ by associating each $2k$-subset $S$ of $V(H_2)$ 
with exactly one vertex $v_S$ of $H_1$ and exactly one vertex 
$w_S$ of $H_3$, and adding edges joining $v_S$ and $w_S$ to every
vertex in $S$. 

We now prove that $\overrightarrow{{\rm diam}}(G_k) \geq 3$. 
Let $D_k$ be an arbitrary orientation of $G_k$. Fix a vertex $v \in V(H_1)$. 
Then $|N^+(v)\cap V(H_2)| + |N^-(v)\cap V(H_2)| = 2k$, and
so $|N^+(v)\cap V(H_2)| \leq k$ or $|N^-(v)\cap V(H_2)| \leq k$.
Without loss of generality we assume $|N^+(v)\cap V(H_2)| \leq k$. 
Then $|V(H_2) - (N^+(v)\cap V(H_2))| \geq 2k$, and so 
there exists a set $S \subseteq V(H_2) - (N^+(v)\cap V(H_2))$
with $|S|=2k$. Consider vertex $w_S$. We claim that
\begin{equation} d_{D_k}(v,w_s) \geq 3. \label{*}
\end{equation}
Indeed, $w_S$ and $v$ are non-adjacent in $G_k$ and no
vertex in $N_{G_k}(v) \cap N_{G_k}(w_S)$ is an out-neighbour
of $v$ in $D_k$. Hence there exists no $(v,w_S)$-path of length 
two in $D_k$, and $(\ref{*})$ follows. 

Now $(\ref{*})$ implies that ${\rm diam}(D_k) \geq 3$, and since $D_k$
was arbitrary, we have $\overrightarrow{{\rm diam}}(G_k) \geq 3$. 

We now determine the minimum degree of $G_k$. Let $n_k = |V(G_k)|$. 
Each vertex in $H_1$ or $H_3$ has degree 
${3k \choose k} +2k-1 = \frac{1}{2}n_k + \frac{k}{2}-1$. 
Every vertex in $H_2$ is adjacent to all vertices in $H_2$ 
except itself, and to 
two thirds of the vertices in $H_1$ and $H_3$, and so 
has degree 
$\frac{4}{3} {3k \choose k} + 3k-1=\frac{2}{3}n_k+k-1$. It follows 
that 
\begin{equation}
\delta(G_k)  = \min\bigl( \frac{1}{2}n_k + \frac{k}{2}-1, 
       \frac{2}{3}n_k + k-1\bigl)
        = \frac{1}{2}n_k + \frac{k}{2}-1. \label{**} 
        \end{equation}
Using Stirling's Formula, we obtain for sufficiently large $k$ that
\begin{equation}
n_k = 2 {3k \choose k} + 3k  \leq 2.5 \frac{(3k)!}{k!(2k)!} <
3\frac{ \bigl(\frac{3k}{e}\bigl)^{3k}\sqrt{6k\pi}    }{ \bigl(\frac{k}{e}\bigl)^{k}\sqrt{2k\pi}  \bigl(\frac{2k}{e}\bigl)^{2k}\sqrt{4k\pi}    }=\frac{3^{3/2}}{2}\biggl(\frac{3^{3}}{2^2}\biggl)^k
\frac{1}{\sqrt{k\pi}}.
          \label{***}  \end{equation}
Taking logarithms of both sides of       (\ref{***}), we get 
\[  \ln n_k    <  \ln \bigl(\frac{3^{3/2}}{2\sqrt{\pi}}\bigl) +k \ln (27/4) - (1/2)\ln k.
           \]         
Substituting this into (\ref{**}) yields  
\[ \delta(G_k) - \frac{1}{2}n_k 
      =\frac{k}{2}-1
     >\frac{\ln n_k-\ln (\frac{3^{3/2}}{2\sqrt{\pi}})}{ 2\ln(27/4) }-1  +\frac{1}{4}\ln k 
   > \frac{\ln n_k}{ 2\ln(27/4) }  \]
since for large $k$ the last term dominates the negative constant.
\end{proof}

It follows from Theorems \ref{theo:sufficent-for-diam2-orient} and
\ref{theo:sharpness-example} that the least minimum degree $\delta_n$ that
guarantees the existence of an orientation of diameter two in a graph 
of order $n$ and minimum degree not less than $\delta_n$ satisfies 
$\delta_n=\frac{1}{2}n + \Theta(\ln n)$. 

%
%



\end{document}